\documentclass[a4paper]{amsart}
\usepackage{amsmath,amsthm,amssymb}
\usepackage[pdftex]{graphicx}
\usepackage{mathrsfs}

\newtheorem{thm}{Theorem}[section]
\newtheorem{lem}[thm]{Lemma}
\newtheorem{cor}[thm]{Corollary}

\theoremstyle{definition}

\newtheorem{example}[thm]{Example}

\newtheorem{defn}[thm]{Definition}

\theoremstyle{remark}

\newtheorem{rem}[thm]{Remark}        

\numberwithin{equation}{section}

\def\XXint#1#2#3{{\setbox0=\hbox{$#1{#2#3}{\int}$}
\vcenter{\hbox{$#2#3$}}\kern-.5\wd0}}

\newcommand{\R}{\mathbb{R}}

\newcommand{\m}{\mathtt{m}}

\newcommand{\Lip}{\mathsf{Lip}}

\newcommand{\di}{\mathsf{Diam}\,}

\newcommand{\vol}{\mathsf{vol}}
\newcommand{\RCD}{\mathsf{RCD}}
\newcommand{\CDC}{\mathsf{CD}}
\newcommand{\ul}{\underline}
\newcommand{\ol}{\overline}
\newcommand{\la}{\left\langle}
\newcommand{\ra}{\right\rangle}
\newcommand{\argmax}{\mathrm{argmax}\,}
\newcommand{\argmin}{\mathrm{argmin}\,}
\newcommand{\KD}{\mathsf{KD}}
\newcommand{\del}{\partial}
\newcommand{\iden}{\mbox{1}\hspace{-0.25em}\mbox{l}}

\newcommand{\calC}{\mathcal{C}}
\newcommand{\calD}{\mathcal{D}}

\newcommand{\calL}{\mathcal{L}}

\newcommand{\calR}{\mathcal{R}}

\newcommand{\bbS}{\mathbb{S}}

\newcommand{\ttT}{\mathtt{T}}

\newcommand{\ttb}{\mathtt{b}}

\newcommand{\sfb}{\mathsf{b}}

\usepackage[abbrev,nobysame]{amsrefs}

\makeatletter
\def\@makefnmark{%
\leavevmode
\raise.9ex\hbox{\check@mathfonts
\fontsize\sf@size\z@\normalfont%
\@thefnmark}%
}
\makeatother

\title{Cheng's Maximal diameter theorem for hypergraphs}
\author{Yu Kitabeppu}
\address[Yu Kitabeppu]{Kumamoto University}
\email{ybeppu@kumamoto-u.ac.jp}
\author{Erina Matsumoto}
\address[Erina Matsumoto]{Kumamoto University}
\email{191d8010@st.kumamoto-u.ac.jp}
\begin{document}
\maketitle
 \begin{abstract}
  We prove that Cheng's maximal diameter theorem for hypergraphs with positive coarse Ricci curvature.   
 \end{abstract}
%
%
\section{Introduction}
 It holds that the diameter of an $n$-dimensional Riemannian manifold with Ricci curvature bounded below by $K(n-1)>0$ is at most $\pi/\sqrt{K}$ by the famous Bonnet-Myers theorem. Moreover, under the same assumption and if $\di M=\pi/\sqrt{K}$, then $M$ is isometric to the $n$ dimensional sphere $\bbS^n(K^{-1/2})$, which is known as Cheng's maximal diameter theorem. These theorem have been extended to many other situations(see for instance \cite{BLsobolev,BQvol,BG,Stmms2,LLY,Limyers,LW,Bon,O,LV,Ohmcp,OSY,IKT} for Bonnet-Myers theorem and \cite{Kemax,CKKLMP,OSYmax,Ohpro,Kumax,Ruan} for rigidity theorem). Here we focus on Bonnet-Myers type and Cheng's maximal diameter theorems for $\CDC$ spaces and $\RCD$ spaces. 
 \par A synthetic notion of Ricci curvature bounded from below and dimension bounded from above on generic metric measure spaces, called the Curvature-Dimension condition $\CDC(K,N)$ ($K\in\R$, $N\in[1,\infty]$), was introduced by Sturm and Lott-Villani \cite{LV,Stmms1,Stmms2}. Though many geometric and analytic properties are proven such spaces, the class of $\CDC$ spaces is seen to be still huge. In fact, Finsler manifolds can become $\CDC$ spaces and the Laplacian and the heat flow with respect to the Cheeger energy on generic $\CDC$ spaces are possibly nonlinear. To restrict $\CDC$ spaces to be more a manageable class, $\RCD$ spaces are defined by satisfying $\CDC$ condition with the infinitesimally Hilbertian condition \cite{AGSRiem,AGMR,EKS,AMSnon}. The heat flow and the Laplacian become linear operators thanks to the infinitesimal Hilbertianity. It is proven that on infinitesimal Hilbertian space, $\CDC$ condition and Bakry-\'Emery type curvature-dimension condition coincides under mild assumptions \cite{AGSbercd}. For finite dimensional $\CDC$ spaces, Bonnet-Myers theorem holds \cite{LV,Stmms2}. On the other hand, the maximal diameter theorem have not been proven on $\CDC$ spaces(under the non-branching assumption, $\CDC$ spaces are $\mathsf{MCP}$ ones. Therefore a partial result is given in \cite{Ohpro}. On the other hand, $\mathsf{MCP}$ condition is not enough to guarantee the space being spherical suspension. see \cite{KeR}). However on $\RCD$ spaces, Cheng's maximal diameter theorem holds \cite{Kemax}. In \cite{Kemax}, the infinitesimal Hilbertianity(or the linearity of the Laplacian) seems to play an important role in the proof.
 \par There exist many notions of Ricci curvature on graphs(for example,  \cite{BHLLMY,Bon,BonS,EM,LLY,Maas,Munch1,Munch2,O,OSY,Schm}). The positivity of lower bound of Ricci curvature often leads the bounds of diameter. And sometimes Cheng's maximal diameter theorem is also given. For instance, the maximal diameter theorem on graphs and directed graphs under positive Ricci curvature in the sense of Lin-Lu-Yau are proven in \cite{CKKLMP,OSYmax}. Note that Laplacian related such notions are linear though sometimes those operator is not symmetric and the corresponding Dirichlet form is nonlocal. Hence can we guess the linearity of the Laplacian needs to guarantee holding the maximal diameter theorem?    
 \par In the present paper, we consider Cheng's maximal diameter theorem under positive bound of coarse Ricci curvature on hypergraphs. The first author, Ikeda, and Takai introduced a Lin-Lu-Yau's type coarse Ricci curvature notion on the class of weighted hypergraphs \cite{IKT}. Under positive Ricci curvature bound, we have Bonnet-Myers type diameter bound. Actually if $\inf_{x\neq y}\ul{\kappa}(x,y)\geq K>0$, $\di V\leq 2/K$ holds. Although our definition of Ricci curvature is based on the resolvent operator associated with the Laplacian that is nonlinear in general, we have the following theorem. 
 \begin{thm}
  Let $H:=(V,E,\omega)$ be a finite weighted hypergraph. Assume that $\inf_{x\neq y}\ul{\kappa}(x,y)\geq K>0$. Accordingly $\di V\leq 2/K$. Moreover assume $\di V=2/K$. Then there exists a pair of points $p,q\in V$ with $d(p,q)=\di V$ such that every point $v\in V$ is on a geodesic path connecting $p$ and $q$. Moreover if $x,y\in V$ are on the same geodesic path connecting $p$ to $q$, then $\ul{\kappa}(x,y)=\ol{\kappa}(x,y)=K$.  
 \end{thm}
%
%
\section{Preliminaries}
\subsection{Monotone operator}
 In this subsection, we give some results of the theory of maximal monotone operators that we need later. See \cite{Mi} for the detailed proofs or more reference. 
 \par Let $H$ be a Hilbert space with the inner product $\la \cdot,\cdot\ra$ and $A$ a multivalued operator, that is, $Ax\subset H$ for any $x$ in the domain of $A$, denoted by $\calD(A)$. We denote the range of $A$ by $\calR(A):=\{y\;;\;y\in Ax\text{ for some }x\in\calD(A)\}$. And we define the inverse of $A$ by $A^{-1}x=\{y\;;\;x\in Ay\}$ with $\calD(A^{-1})=\calR(A)$. 
 We say that $A$ is \emph{monotone} if 
 \begin{align}
  \la x'-y',x-y\ra\geq 0\notag
 \end{align} 
 holds for any $x,y\in\calD(A)$ and any $x'\in Ax$, $y'\in Ay$. A monotone operator $A$ is called an \emph{$\m$-monotone operator} if $\calR(I+\lambda A)=H$ for any $\lambda>0$, where $I$ is the identity operator on $H$. It is known that $A$ is an $\m$-monotone operator if and only if $A$ is a maximal monotone one. Here a monotone operator $A$ is said to be maximal if any monotone operator $B$ satisfying $\calD(A)\subset \calD(B)$ and $Ax\subset Bx$ for any $x\in\calD(A)$ coincides $A$.   
 \par From now on, we always assume that $A$ is an $\m$-monotone operator. It is known that $Ax$ is a closed convex subset in $H$ for any $x\in \calD(A)$. Hence we find a unique element $\tilde{x}=\argmin\{\Vert x'\Vert\;;\;x'\in Ax\}$. We call $\tilde{x}$ the \emph{canonical restriction} of $A$ and denote it by $A^0x$. Since $\calR(I+\lambda A)=H$, we are able to define the resolvent operators $J_{\lambda}:=(I+\lambda A)^{-1}$. One notable property of $J_{\lambda}$ is that is actually single-valued. We also know the following result. 
  \begin{thm}
   Let $A$ be an $\m$-monotone operator. Then we have a unique solution $u(t)$ to 
   \begin{align}
    \frac{d}{dt}u(t)\in -Au(t),\;\text{a.e. }t>0\quad u(0)=x\in\calD(A).\notag
   \end{align}
   Moreover it is equivalent to the existence of a unique solution to 
   \begin{align}
    \frac{d^+}{dt}u(t)=-A^0u(t),\;t\geq 0\quad u(0)=x\in\calD(A).\notag
   \end{align} 
   These two solutions are absolutely continuous and coincide with each other. 
  \end{thm}
 By the result above, we have the contraction semigroup with the generator $A^0$, that is, there exists a family of single-valued operators $\{T_t\}_{t>0}$ with $\calD(T_t)=\calD(A^0)$ such that 
 \begin{itemize}
  \item $T_0=I$, $T_{s+t}=T_sT_t$,
  \item For $x\in \calD(T_t)$, $T_tx$ is strongly continuous with respect to $t$, 
  \item $\Vert T_tx-T_ty\Vert\leq \Vert x-y\Vert$, 
  \item \begin{align}
   \lim_{h\downarrow 0}\frac{1}{h}(T_hx-x)=-A^0x\notag
  \end{align}
  holds for any $x\in \calD(A^0)$. 
 \end{itemize}
 In general $T_t$ is not a linear operator. 
 \begin{example}
  Let $E:H\rightarrow\R\cup\{\infty\}$ be a convex and lower semi-continuous functional on a real Hilbert space $(H,\la\cdot,\cdot\ra)$. Set $\calD(E):=\{f\in H\;;\;E(f)<\infty\}$. $h\in H$ is a sub-differential of $E$ at $f$ provided that 
  \begin{align}
   E(g)-E(f)\geq \la h,g-f\ra\notag
  \end{align}
  holds for any $g\in H$. The set of all sub-differentials of $E$ at $f$ is denoted by $\del E(f)$. Then it is known that $\del E:H\rightarrow 2^H$ is an $\m$-monotone operator defined on $\calD(\del E):=\{f\in\calD(E)\;;\;\del E(f)\neq\emptyset\}$. 
  \par For this case, we have another representation of $J_{\lambda}$. It actually satisfies 
  \begin{align}
   J_{\lambda}f=\argmin\left\{\frac{\Vert f-g\Vert^2}{2\lambda}+E(g)\;;\;g\in \calD(E)\right\}.\notag
  \end{align}
 \end{example}
\subsection{Hypergraph}
A hypergraph $H=(V,E)$ consists of a vertex set $V$ and a hyperedge set $E\subset 2^V$. Unlike in the case of a graph, $e\in E$ can include more than two vertices. We use the notation $x\sim y$ for $x,y\in V$ provided $x,y$ is included in a same hyperedge $e\in E$. For given two vertices $x,y\in V$, we define the distance between them by 
\begin{align}
 d(x,y):=\min\left\{n\;;\;x=x_0,\,x_i\sim x_{i+1},\,x_n=y\;\text{for }i=1,\cdots,n-1\right\}.\notag
\end{align}
A sequence of vertices $\{x_i\}_{i=0}^n$ that realizes the distance $d(x_0,x_n)$ is called a \emph{geodesic path}. 
When $\vert V\vert<\infty$, $H=(V,E)$ is called a finite hypergraph. A weight function $\omega:E\rightarrow \R_{>0}$ is denoted by $\omega_e$ or $\omega(e)$ for $e\in E$. The \emph{weighted degree} $d_x$ at $x\in V$ is defined by $d_x:=\sum_{e\ni x}\omega_e$. Set a matrix $D=(D_{u,v})_{u,v\in V}\in \R^{V\times V}$ by 
\begin{align}
 D_{u,v}:=\begin{cases}
  d_u&\text{if }u=v,\\
  0&\text{otherwise},
 \end{cases}\notag
\end{align}
and call it the \emph{(weighted) degree matrix}. 
In this article, we consider only finite weighted hypergraphs. 
\smallskip
\par Since $\vert V\vert<\infty$, a function $f:V\rightarrow\R$ is identified with a column vector $f\in \R^V$. An inner product for $f,g\in \R^V$ is defined by 
\begin{align}
 \la f,g\ra:=\sum_{x\in V}f(x)g(x)d_x^{-1}.\notag
\end{align}
We define a function $\delta_v\in \R^V$ for $v\in V$ by 
\begin{align}
 \delta_v(u)=
 \begin{cases}
  1&\text{ if }u=v,\\
  0&\text{ otherwise}.
 \end{cases}\notag
\end{align}
Given $f\in\R^V$, the \emph{base polytope $B_e$} for $e\in E$ is defined as 
\begin{align}
 B_e=\mathrm{Conv}\left\{\delta_u-\delta_v\;;\; u,v\in e\right\}\subset \R^V.\notag
\end{align}
We say a function $f:V\rightarrow \R$ \emph{weighted $K$-Lipschitz} for $K>0$ if $\la f,\delta_x-\delta_y\ra\leq Kd(x,y)$ holds for any $x,y\in V$ and denote the set of weighted $K$-Lipschitz functions by $\Lip_{\omega}^K(V)$. 
\bigskip
\par We define the \emph{Laplacian} on the functions defined on $V$ as a multivalued operator by 
\begin{align}
 \calL f:=\left\{\sum_{e\in E}\omega_e\sfb_e^{\ttT}(D^{-1}f)\sfb_e\;;\;\sfb_e\in\argmax_{\sfb\in B_e}{\sfb^{\ttT}(D^{-1}f)}\right\},\notag
\end{align}
here $\sfb^{\ttT}$ is the transpose vector of $\sfb$. 
The following theorem is known. 
 \begin{thm}{\cite{IMTY,Y}}
  The Laplacian $\calL$ is an $\m$-monotone operator on the Hilbert space $(\R^V,\la\cdot,\cdot\ra)$. 
 \end{thm}
Hence we define the resolvent operator $J_{\lambda}:\R^V\rightarrow \R^V$ and the contraction semigroup $h_t:\R^V\rightarrow \R^V$, that are both single-valued. We call $\{h_t\}_{t>0}$ the \emph{heat semigroup} or \emph{heat flow}. One holds the following property, that is, 
\begin{align}
 \calL(af)=a\calL f,\;J_{\lambda}(af)=aJ_{\lambda}f,\;h_t(af)=ah_t(f)\label{eq:homoge}
\end{align}
for $f\in\R^V$, $a\in\R$. 
 \begin{rem}
  Actually $\calL$ is defined as the sub-differential of a functional on $\R^V$. Let $n=\vert V\vert$, $N:=\vert E\vert$. Define a functional 
  \begin{align}
   E(f):=\frac{1}{2}\sum_{e\in E}\omega_e\left(\max_{u\in e}D^{-1}f(u)-\min_{v\in e}D^{-1}f(v)\right)^2.\notag
  \end{align}
  For any $f'=\sum_{e\in E}\omega_e\sfb_e^{\ttT}(D^{-1}f)\sfb_e\in\calL f$, $\sfb_e\in\argmax_{\sfb\in B_e}\sfb^{\ttT}(D^{-1}f)$, we denote it by $f'=BWB^{\ttT}\tilde f$, where $B=[\sfb_{e_1},\cdots,\sfb_{e_N}]\in \R^{n\times N}$, $W=\mathrm{diag}(\omega_{e_1},\cdots,\omega_{e_N})\in \R^{N\times N}$, $\tilde f=D^{-1}f\in\R^V=\R^n$. 
  \begin{align}
   E(f)=\frac{1}{2}\sum_{e\in E}\omega_e\left(\sfb_e^{\ttT}(D^{-1}f)\right)^2=\frac{1}{2}(\tilde f)^{\ttT}BWB^{\ttT}\tilde f=\frac{1}{2}\la BWB^{\ttT}\tilde f,f\ra=\frac{1}{2}\la f',f\ra\notag
  \end{align}
  holds. The continuity and the convexity of $E$ are clear. For $k,h\in\R^E$, we define $\la h,k\ra_W:=h^{\ttT}Wk$. It is an inner product that satisfies $\la f',f\ra=\la B^{\ttT}\tilde f,B^{\ttT}\tilde f\ra_W$. Then for any $g\in \R^V$, $g'=B_1WB_1^{\ttT}\tilde g\in \calL g$, it holds 
  \begin{align}
   \la f',g-f\ra&=\la f',g\ra-\la f',f\ra=\la B^{\ttT}\tilde f,B^{\ttT}\tilde g\ra_W-\la f',f\ra\notag\\
   &\leq \Vert B^{\ttT}\tilde f\Vert_W\Vert B^{\ttT}\tilde g\Vert_W-\la f',f\ra\notag\\
   &\leq \frac{1}{2}\Vert B^{\ttT}\tilde f\Vert_W^2+\frac{1}{2}\Vert B^{\ttT}\tilde g\Vert_W^2-\la f',f\ra\notag\\
   &\leq \frac{1}{2}\Vert B_1^{\ttT}\tilde g\Vert^2_W-\frac{1}{2}\Vert B^{\ttT}\tilde f\Vert_W^2\notag\\
   &=E(g)-E(f).\notag
  \end{align}
  This means $f'\in\del E(f)$. Accordingly $\calL f\subset \del E(f)$. Since both $\calL$ and $\del E$ are maximal monotone operators, $\calL=\del E$.  
 \end{rem}
 $\iden\in \R^V$ is defined by $\iden(x)=1$ for any $x\in V$. 
\begin{lem}
 For any $f\in \R^V$ and any $f'\in \calL f$, 
 \begin{align}
  \la f',D\iden\ra=0.\notag
 \end{align}
\end{lem}
 \begin{proof}
  By the definition of the sub-differentials, we have 
  \begin{align}
   0=E(f+D\iden)-E(f)\geq \la f',D\iden\ra,\notag\\
   0=E(f-D\iden)-E(f)\geq \la f',-D\iden\ra.\notag
  \end{align}
  Combining these inequalities leads the consequence. 
 \end{proof}
\subsection{Coarse Ricci curvature on hypergraphs} 
Let $H=(V,E,\omega)$ be a weighted finite hypergraph. We define the new metric $\KD_{\lambda}(x,y)$ on $V$, which plays a role of the $L^1$-Wasserstein distance between random walks with the initial distributions $\delta_x$ and $\delta_y$ respectively. 
 \begin{defn}[Nonlinear Kantorovich difference \cite{IKT}]
  Given $\lambda>0$, $x,y\in V$. The \emph{$\lambda$-nonlinear Kantorovich difference} $\KD_{\lambda}(x,y)$ is defined as 
  \begin{align}
   \KD_{\lambda}(x,y):=\sup\left\{\la J_{\lambda}f,\delta_x-\delta_y\ra\;;\;f\in\Lip_{\omega}^1(V)\right\}.\notag
  \end{align}
 \end{defn}
 \begin{rem}
  Each $\KD_{\lambda}$ is a metric on $V$(\cite{IKT}). They are defined to imitate the Kantorovich-Rubinstein duality formula. We take a devious route not to tackle a problem to determine what is the canonical random walk on a hypergraph. 
  Relations between $\KD_{\lambda}$ and $\KD_{\mu}$ for $\lambda,\mu>0$ are still unclear. 
 \end{rem}
 We define a coarse Ricci curvature as a similar way to Lin-Lu-Yau's coarse Ricci curvature on graphs. 
 \begin{defn}[Coarse Ricci curvature \cite{IKT}]
  $\lambda$-coarse Ricci curvature $\kappa_{\lambda}(x,y)$ ($x,y\in V$) is defined by 
  \begin{align}
   \kappa_{\lambda}(x,y):=1-\frac{\KD_{\lambda}(x,y)}{d(x,y)}.\notag
  \end{align}
  The upper coarse Ricci curvature $\ol{\kappa}$ is defined by 
  \begin{align}
   \ol{\kappa}(x,y):=\limsup_{\lambda\rightarrow +0}\frac{\kappa_{\lambda}(x,y)}{\lambda}.\notag
  \end{align}
  Similarly, the lower coarse Ricci curvature $\ul{\kappa}$ is defined by 
  \begin{align}
   \ul{\kappa}(x,y):=\liminf_{\lambda\rightarrow+0}\frac{\kappa_{\lambda}(x,y)}{\lambda}.\notag
  \end{align}
 \end{defn}
 \begin{rem}
  For a positive constant $a>0$, hypergraphs $H=(V,E,\omega)$ and $H'=(V,E,a\omega)$ have the same curvature. Indeed, $\la\cdot,\cdot\ra'$, $J_{\lambda}'$, $\KD_{\lambda}'$ and $\Lip_{a\omega}^1(V)$ is the inner product, the resolvent operator, $\lambda$-nonlinear Kantorovich difference and the set of weighted 1-Lipschitz functions on $H'$ respectively. It is easy to verify $J_{\lambda}f=J_{\lambda}'f$ and $f\in\Lip_{\omega}^1(V)$ iff $af\in \Lip_{a\omega}^1(V)$. Therefore, for $f'=af\in\Lip_{a\omega}^1(V)$ with $f\in \Lip_{\omega}^1(V)$, we have 
  \begin{align}
  \la J_{\lambda}'(f'),\delta_x-\delta_y\ra'&=a^{-1}\la J_{\lambda}(af),\delta_x-\delta_y\ra\notag\\
  &=\la J_{\lambda}f,\delta_x-\delta_y\ra.\notag
  \end{align}
  Hence $\KD'_{\lambda}(x,y)=\KD_{\lambda}(x,y)$. Accordingly $\lambda$-coarse Ricci curvature coincides. 
 \end{rem}
 We have the following formula.  
 \begin{lem}
  Let $H=(V,E,\omega)$ be a weighted hypergraph. Then 
  \begin{align}
   \ol{\kappa}(x,y)\leq \frac{1}{d(x,y)}\calC(x,y),\label{eq:anotherrepolkappa}
  \end{align}
  where 
  \begin{align}
   \calC(x,y)=\inf\left\{\la \calL^0 f,\delta_x-\delta_y\ra\;;\;f\in \Lip_{\omega}^1(V),\, \la f,\delta_x-\delta_y\ra=1 \right\}.\notag
  \end{align}
 \end{lem}
 \begin{lem}\label{lem:Lapbound}
  Let $f\in \Lip_w^1(V)$ be a weighted 1-Lipschitz function and $f'\in \calL f$. Then $\vert f'(x)\vert\leq d_x$ holds for any $x\in V$.  
 \end{lem}
\begin{thm}[Bonnet-Myers type theorem]\label{thm:BMthm}
 Let $H=(V,E,\omega)$ be a weighted hypergraph with positive upper coarse Ricci curvature bound $\inf_{u\neq v}\ol{\kappa}(u,v)\geq K>0$. Then $\di H\leq 2/K$.  
\end{thm}
 The proof is already given in \cite{IKT}. Since some techniques used in the proof are useful even for the main theorem, we show the outline of the proof.   
 \begin{proof}
  Let $p,q\in V$ be points such that $d(p,q)=\di H$. For any $f\in \Lip_{\omega}^1(V)$, $\vert\calL^0f(x)\vert\leq d_x$ by Lemma \ref{lem:Lapbound}. Then if $\la f,\delta_p-\delta_q\ra=d(p,q)$, 
  \begin{align}
   K\leq \ol{\kappa}(p,q)\leq \frac{\la \calL^0f,\delta_p-\delta_q\ra}{d(p,q)}\leq \frac{\la \vert \calL^0f\vert,\delta_p+\delta_q\ra}{\di H}\leq \frac{2}{\di H}.\notag
  \end{align} 
 \end{proof}
 \begin{rem}
  Since $\ul{\kappa}(u,v)\leq \ol{\kappa}(u,v)$ holds for any $u,v\in V$, the assumption in Theorem \ref{thm:BMthm} can be replaced by $\inf_{u\neq v}\ul{\kappa}(u,v)\geq K$ instead of the lower bound of upper coarse Ricci curvature. 
 \end{rem}
%
%
\section{Proof of main theorem}
We prove the following theorem. 
 \begin{thm}\label{thm:mainCheng}
  Let $H=(V,E,\omega)$ be a weighted hypergraph with positive lower coarse Ricci curvature bound $\inf_{u\neq v}\ul{\kappa}(u,v)\geq K>0$. Assume $\di H=2/K$. Then there exists a pair of points $p,q\in V$ such that every point $x\in V$ is on a geodesic path from $p$ to $q$. Moreover if $x,y\in V$ are on the same geodesic path. Then $\ul{\kappa}(x,y)=\ol{\kappa}(x,y)\equiv K$. 
 \end{thm}
 \begin{proof}
  Let $p,q\in V$ be points such that $d(p,q)=\di H=2/K=:L$. Define functions $\rho_p$, $\rho_q$, $R:V\rightarrow \R$ by 
  \begin{align}
   \rho_p(x):=d_xr_p(x):=d_xd(p,x), \;\rho_q(x):=d_xd(x,q),\; R(x):=Ld_x.\notag
  \end{align}
  Put $\calL^0\rho_p=\sum_{e}\omega_e\ttb_e\ttb_e^{T}(r_p)$. By the proof of Theorem \ref{thm:BMthm}, we have $\vert\la \calL^0\rho_p,\delta_q-\delta_p\ra\vert=2$. Since $\rho_p(p)=\min_{x}\rho_p(x)$, we have $\calL^0\rho_p(p)=\sum_{e\ni p}\omega_e(-1)=-d_p$. Combining these two facts, we obtain 
  \begin{align}
   2=\la \calL^0\rho_p,\delta_q-\delta_p\ra=\la \calL^0\rho_p,\delta_q\ra+1\;\Leftrightarrow \la \calL^0\rho_p,\delta_q\ra=1.\notag
  \end{align}
  By the same way, we also have $\la \calL^0\rho_q,\delta_p\ra=-\la \calL^0\rho_q,\delta_q\ra=1$. 
  \par Suppose there exists a point $v\in V$ such that 
  \begin{align}
   f(v):=\rho_p(v)+\rho_q(v)-R(v)\geq d_v>0.\label{eq:key0}
  \end{align} 
  By (\ref{eq:anotherrepolkappa}), 
  \begin{align}
   K\leq \ul{\kappa}(p,v)\leq \frac{\la \calL^0\rho_p,\delta_v-\delta_p\ra}{d(p,v)}\Leftrightarrow Kr_p(v)\leq \la \calL^0\rho_p,\delta_v\ra+1.\notag
  \end{align} 
  Multiplying $d_v$ leads 
  \begin{align}
   K\rho_p(v)\leq \calL^0\rho_p(v)+d_v.\label{eq:key1}
  \end{align}
  By a similar calculation, we have 
  \begin{align}
   K\rho_q(v)\leq \calL^0\rho_q(v)+d_v.\label{eq:key2}
  \end{align}
  Combining (\ref{eq:key0}), (\ref{eq:key1}), and (\ref{eq:key2}), we obtain 
  \begin{align}
   \calL^0\rho_p(v)+\calL^0\rho_q(v)+2d_v&\geq K\left(\rho_p(v)+\rho_q(v)\right)=K\left(R(v)+f(v)\right)\notag\\
   &=2d_v+\frac{2}{L}f(v).\notag
  \end{align}
  Set $A:=\{v\;;\; f(v)>0\}$. Summing up $v\in V$ leads, 
  \begin{align}
   \sum_{v\in V}\left(\calL^0 \rho_p(v)+\calL^0\rho_q(v)\right)\geq \frac{2}{L}\sum_{v\in V}f(v)\geq \frac{2}{L}\sum_{x\in A}f(v)=\frac{2}{L}\vol(A)>0.\notag
  \end{align}
  On the other hand 
  \begin{align}
   \sum_{v\in V}\left(\calL^0\rho_p(v)+\calL^0\rho_q(v)\right)=\la \calL^0\rho_p,D\iden\ra+\la\calL^0\rho_q,D\iden\ra=0.\notag
  \end{align}
  This contradicts. Hence each point is on a geodesic path connecting from $p$ to $q$. 
  
  Take a pair of points $v,w$ on a same geodesic path $\gamma$ connecting from $p$ to $q$. With out loss of generality, we may assume $v=\gamma_k$, $w=\gamma_l$ with $k<l$. Note that 
  \begin{align}
   \la \rho_p,\delta_w-\delta_v\ra=d(v,w),\notag
  \end{align}
  since $v,w$ are on $\gamma$. By (\ref{eq:anotherrepolkappa}), we have 
  \begin{align}
   &Kd(p,v)\leq \la \calL^0\rho_p,\delta_v-\delta_p\ra,\notag\\
   &Kd(v,w)\leq \la \calL^0\rho_p,\delta_w-\delta_v\ra,\notag\\
   &Kd(w,q)\leq \la \calL^0\rho_p,\delta_q-\delta_w\ra. \notag
  \end{align}
  Adding these inequalities, we have 
  \begin{align}
   Kd(p,q)=K\left(d(p,v)+d(v,w)+d(w,q)\right)\leq \la \calL^0\rho_p,\delta_q-\delta_p\ra=Kd(p,q).\notag
  \end{align}
  Hence all inequalities are actually equalities. Therefore 
  \begin{align}
   K\leq \ul{\kappa}(v,w)\leq \ol{\kappa}(v,w)\leq \frac{\la\calL^0\rho_p,\delta_w-\delta_v\ra}{d(v,w)}=K.\notag
  \end{align}
 \end{proof}
On Riemannian manifolds, Lichnerowicz-Obata's theorem is also known. It states that the first non-zero eigenvalue $\lambda_1$ of the Laplacian on $n$-dimensional Riemannian manifold $M$ with $\mathrm{Ric}_g\geq (n-1)$ satisfies $\lambda_1\geq n$ and $\lambda_1=n$ if and only if $M=\bbS^n(1)$. We have a weaker theorem, whose proof is similar to \cite{OSYmax}. 
 \begin{cor}
  Under the same assumption as in Theorem \ref{thm:mainCheng}, then $\lambda_1=K$. 
 \end{cor}
 \begin{proof}
  We have $\lambda_1\geq K$ by Theorem 5.1 in \cite{IKT}. By the proof of Theorem \ref{thm:mainCheng}, inequalities (\ref{eq:key1}) and (\ref{eq:key2}) are actually equalities, that is, $K\rho_p(v)=\calL^0\rho_p(v)+d_v$. Set $f:=\rho_p-K^{-1}D\iden$. Since $\calL^0(\rho_p-K^{-1}D\iden)=\calL^0\rho_p$ (Lemma 2.3 in \cite{IKT}), then 
  \begin{align}
   \calL^0f(v)=\calL^0(\rho_p-K^{-1}D\iden)(v)=\calL^0(\rho_p)(v)=K\rho_p(v)-d_v=Kf(v).\notag
  \end{align}
  Hence $\lambda_1=K$. 
 \end{proof}
 \begin{rem}
  Since (\ref{eq:key2}) is also equality, $\rho_q-K^{-1}D\iden$ is also an eigenfunction for $\lambda_1=K$. However due to the nonlinearity of $\calL^0$, the set of eigenfunctions is not a linear space.  
 \end{rem}
%
%
\section{Examples}
\begin{example}
Let $V:=\{p,q,v_1,v_2\}$ be a vertex set and $E_1:=\{e_i\}_{i=1}^7$ a hyperedge set, where 
\begin{align}
 e_1=pv_1,\;e_2=pv_2,\;e_3=v_1v_2,\;e_4=v_1q,\;e_5=v_2q,\;e_6=pv_1v_2,\;e_7=v_1v_2q.\notag
\end{align}
Define a weight function $\omega:E_1\rightarrow\R_{>0}$ by $\omega_e\equiv 1$. Then the weighted hypergraph $H_1=(V,E_1,\omega)$ satisfies $\inf_{x\neq y}\ul{\kappa}(x,y)\geq 1$, and $\di V=2$. Thus $H_1$ satisfies the assumption of Theorem \ref{thm:mainCheng}. Actually $v_1,v_2$ is on a geodesic from $p$ to $q$, and $\ul{\kappa}(p,v_i)=\ul{\kappa}(v_i,q)=1$ is satisfied. 
\begin{figure}[h]
       \centering
        \includegraphics[width=8.5cm]{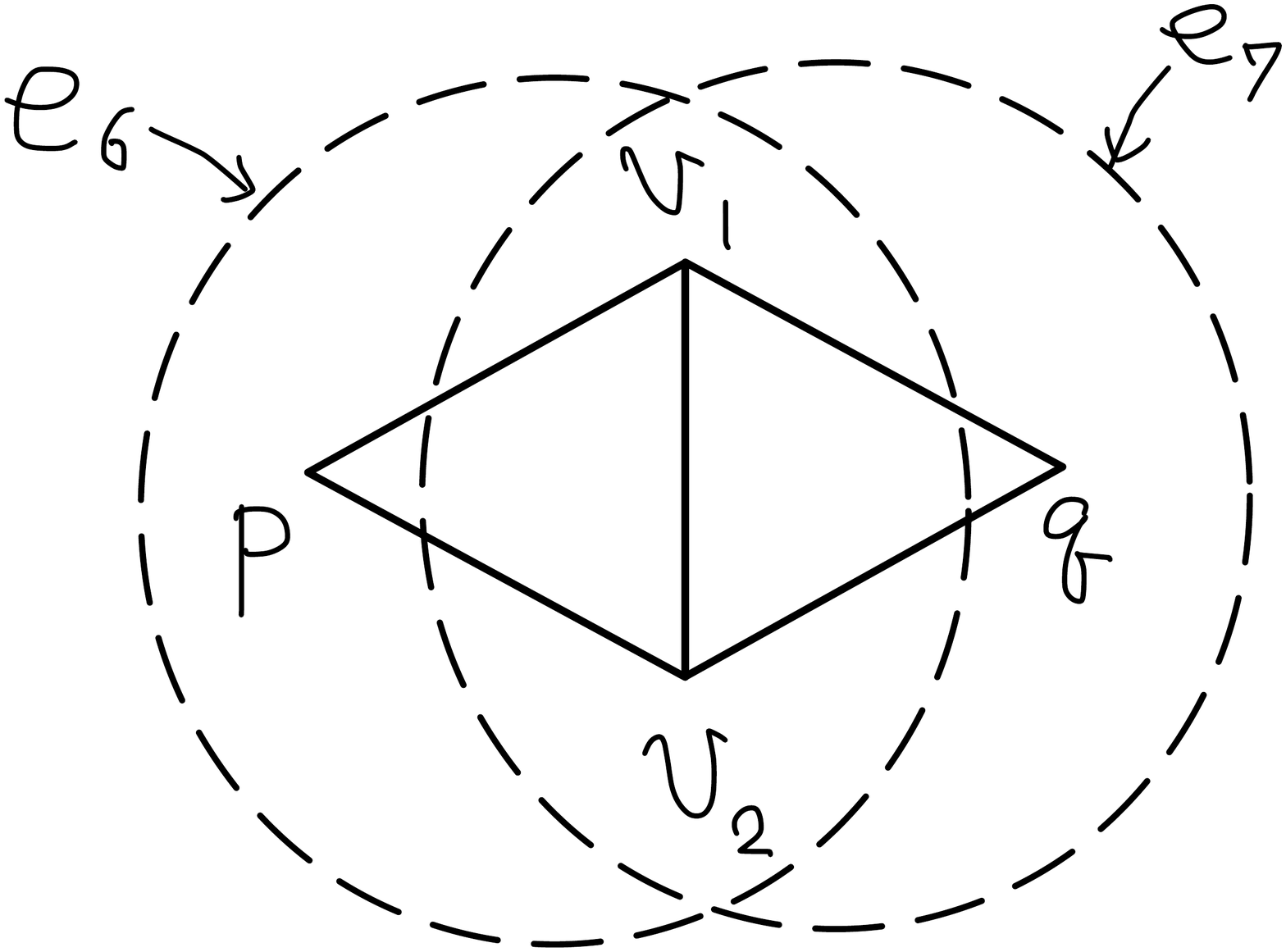}
       \caption{\footnotesize{$H_1=(V,E_1)$}}
\end{figure}
\end{example}
 \begin{example}
  Let $V:=\{p,q,v_1,v_2\}$ be a vertex set and $E_2:=\{e_1,e_2\}$ a hyperedge set, where 
\begin{align}
 e_1=pv_1v_2,\;e_2=v_1v_2q.\notag
\end{align}
Define a weight function $\omega:E_2\rightarrow\R_{>0}$ by $\omega_e\equiv 1$. Then the weighted hypergraph $H_2=(V,E_2,\omega)$ satisfies $\inf_{x\neq y}\ul{\kappa}(x,y)\geq 1$, and $\di V=2$. Thus $H_2$ satisfies the assumption of Theorem \ref{thm:mainCheng}. Actually $v_1,v_2$ is on a geodesic from $p$ to $q$, and $\ul{\kappa}(p,v_i)=\ul{\kappa}(v_i,q)=1$ is satisfied. 
\begin{figure}[h]
       \centering
        \includegraphics[width=8.5cm]{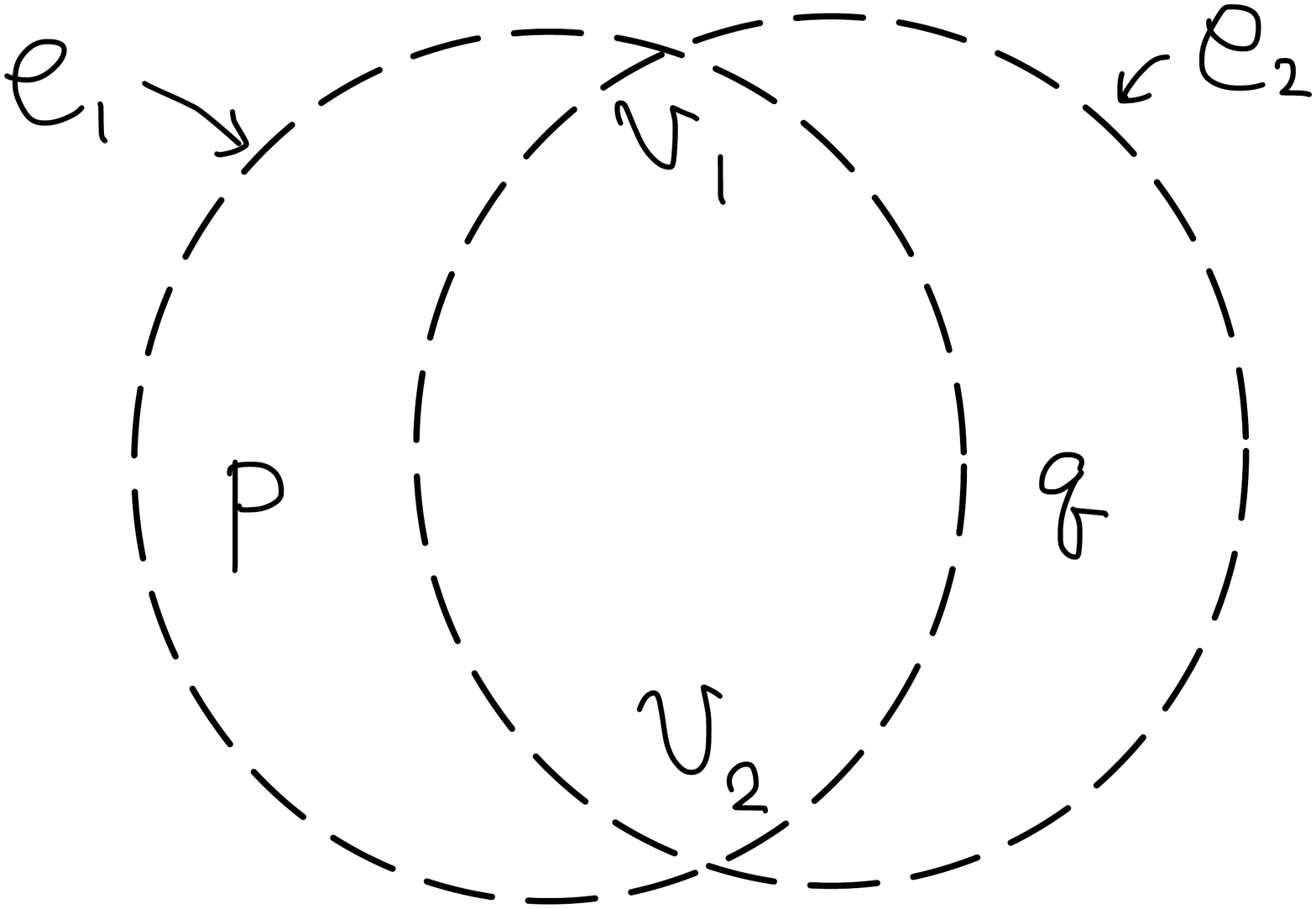}
       \caption{\footnotesize{$H_2=(V,E_2)$}}
\end{figure}

 \end{example}

\section*{Acknowledgement}
The authors would like to thank Professors Shin-ichi Ohta and Yohei Sakurai for their helpful comments and fruitful discussions. The first author is supported by Grant-in-Aid for Young Scientists JP18K13412. 

\end{document}